\documentclass[11pt,english]{article}
\usepackage[T1]{fontenc}
\usepackage[utf8]{luainputenc}
\usepackage{xcolor}
\usepackage{pdfcolmk}
\usepackage{amsthm}
\usepackage{amsmath}
\usepackage{stmaryrd}
\usepackage{stackrel}
\usepackage{setspace}
\PassOptionsToPackage{normalem}{ulem}
\usepackage{ulem}

\makeatletter

\providecolor{lyxadded}{rgb}{0,0,1}
\providecolor{lyxdeleted}{rgb}{1,0,0}

\numberwithin{figure}{section}
\numberwithin{equation}{section}
\numberwithin{table}{section}

  \input amssymb.sty
\usepackage{etoolbox}
\patchcmd{\thebibliography}{\section*}{\section}{}{}

\usepackage{algpseudocode} \usepackage{algorithm}
\usepackage{algpseudocode,algorithm}
\usepackage{lipsum}
\usepackage{caption}
\usepackage{graphics}

\usepackage{amsmath}
\usepackage{amssymb}
\usepackage[all]{xy}

\usepackage{epsfig}\usepackage{youngtab}

\newcommand{\ef}{\end{equation}}
\chardef\bslash=`\\ 

\newdir{:=}{{}}
\newcommand*\colvec[3][]{
    \begin{pmatrix}\ifx\relax#1\relax\else#1\\\fi#2\\#3\end{pmatrix}
}

\newtheorem*{thm*}{Theorem}

\newtheorem{lem}{Lemma}[section]
\newtheorem*{lem*}{Lemma}
\newtheorem{corl}{Corollary}[lem]
\newtheorem*{corl*}{Corollary}

\newtheorem{prop}{Proposition}[section]
\newtheorem{prop*}{Proposition}

\theoremstyle{definition}
\newtheorem{defn}{Definition}[section]
\newtheorem{examp}{Example}
\newtheorem*{examp*}{Example}
\newtheorem*{remark*}{Remark}
\newtheorem*{CC*}{Crossover Conjecture}
\newtheorem*{Note*}{Note}
\newtheorem*{defn*}{Definition}

 \theoremstyle{remark}

 \renewcommand{\sectionmark}[1]{}

\newcommand{\la}{\langle}
\newcommand{\ra}{\rangle}

\newcommand{\defect}{\operatorname{def}}

\newcommand{\hub}{\operatorname{hub}}
\newcommand{\cont}{\operatorname{cont}}

\renewcommand{\a}{\alpha}

\makeatother

\usepackage{babel}

\begin{document}

\vspace{-0.8cm}
\begin{center}
\Large \textbf{\noindent
External vertices for crystals of type A}
\\
\vspace{0.5cm}
\normalsize
\normalsize{
 Ola Amara-Omari$^1$,  Mary Schaps$^2$
} \\

\textit{\footnotesize
$^1$ Partially supported by Ministry of Science, Technology and Space fellowship , at Bar-Ilan University, Ramat-Gan, Israel
olaomari77@hotmail.com\\
$^2$ Bar-Ilan University, Ramat-Gan, Israel,
mschaps@mathiu.ac.il\\
}
\vspace{5mm}

\end{center}

 \begin{abstract} 
We demonstrate that there are only a finite number of Moria equivalence classes of blocks of cyclotomic Hecke algebras of type A, by combining some combinatorics with the Chuang-Rouquier categorification for the simple modules of the cyclotomic Hecke algebras of type $A$ and rank $e$.  This is an extension of a proof for symmetric groups of a conjecture known as Donovan's conjecure.

The blocks of cyclotomic Hecke algebras for a given dominant integral weight $\Lambda$ correspond to the weights $P(\Lambda)$ of a highest weight representation with highest weight $\Lambda$. We connect these weights into a graph we call the reduced crystal $\hat P(\Lambda)$, in which vertices are connected by $i$-strings. We show that a vertex is  $i$-external for a residue $i$ if the defect is less than the absolute value of the  $i$-component of the hub.  
We demonstrate the existence of a bound  on the degree after which all vertices of a given defect $d$ are external in at least one $i$-string.  For $e=2$, we calculate an approximation to this bound.
\end{abstract}

\section{INTRODUCTION}

For an affine Lie algebra of type A and rank $e$, we take a dominant integral weight $\Lambda$ and let $V(\Lambda)$ with the corresponding  highest weight module with set of weights $P(\Lambda)$.  This can be made into a graph   $\hat P(\Lambda)$ by adding edges labelled by residues $i=0,1,\dots, e-1$ determined by the operators $\tilde e_i$ and $\tilde f_i$ of the corresponding Kashiwara crystal $B(\Lambda)$. The Weyl group $W$ of the affine Lie algebra acts on this graph by reflecting $i$-strings. To each vertex there correspond various numerical invariants, including a defect, a content, a degree and a hub, whose definitions will be given in \S2.

In this paper, we first give a numerical criterion for a vertex of defect $d$ to be at the end  of an $i$-string. For any positive integer $n$, there are only a finite number of vertices of degree $\leq n$.  Two vertices at the two ends of an $i$-string are mapped into each other by the action of the Weyl group, one end being the high-degree end and the other the low-degree end.   We then show that there are only a finite number of vertices of any defect $d$ which are not at the high-degree end of an $i$-string for some $i$.  Combining this with Chuang-Rouquier categorification of the cyclotomic Hecke algebras
$H^\Lambda_n$, we conclude that there are only a finite number of Morita equivalence classes of blocks of cyclotomic Hecke algebras, proving a version of Donovan's conjecture for the cyclotomic blocks in type $A$.

For the case  $e=2$, we give explicit formulae for all the invariants and an explicit bound on the degree after which all cyclotomic Hecke algebra blocks are Morita equivalent to blocks of lower degree.

\section{DEFINITIONS AND NOTATION}

Let $\mathfrak g$ be the affine Lie algebra $A^{(1)}_{\ell}$ as in \cite{Ka} and let $e=\ell+1$. Let $C$ be the Cartan matrix, and $\delta$ the null root. Let $\Lambda$ be a dominant integral weight, let $V(\Lambda)$ be a highest weight module with that 
highest weight, and let $P(\Lambda)$ be the set of weights of $V(\Lambda)$. Let $Q$ be the $\mathbb Z$-lattice generated by the simple roots,
\[
\a_0,\dots,\a_{\ell}.
\]
Let  $Q_+$ be the subset of $Q$ in which all coefficients are non-negative. 

The weight space $P$ of the affine Lie algebra has two different bases.  One is given by the fundamental weights together with the null root, $\Lambda_0,\dots, \Lambda_{e-1}, \delta$, 
and one is given by $\Lambda_0, \a_0,\dots,\a_{e-1}$.  We will usually use the first basis for our weights.  In type A, if $\Lambda=a_0\Lambda_0+a_1\Lambda_1+\dots a_\ell\Lambda_\ell$, we let $r=a_0+a_1+ \dots +a_\ell$ and call it the level.

The Cartan matrix in type $A$ is symmetric, and thus induces, through duality,  a symmetric product on the weight space.  As in  \cite{Kl} we define the defect of a weight $\lambda=\Lambda -\alpha$ by 

\[
\defect(\lambda)=\frac{1}{2}((\Lambda \mid \Lambda)-  (\lambda \mid \lambda))=(\Lambda \mid \alpha)-\frac{1}{2}(\alpha \mid \alpha).
\]

Since we are in a highest weight module, we always have $(\Lambda \mid \Lambda) \geq (\lambda \mid \lambda)$, and the defect is in fact an integer for the affine Lie algebras of type $A$ treated in this paper. The weights of defect $0$ are those lying in the Weyl group orbit of $\Lambda$.

 Every weight $\lambda \in P(\Lambda)$ has the form $\Lambda-\alpha$, for $\alpha \in Q_+$. If $\alpha=\sum_{i=0}^\ell c_i \alpha_i$, for all $c_i$ non-negative, then the vector 
\[
\cont(\lambda)=(c_0.\dots,c_\ell)
\]
is called the content of $\lambda$.

Define
\[
\max (\Lambda)=\{\lambda \in P(\Lambda) \mid \lambda + \delta \not\in P(\Lambda)\},
\] 
and by \cite{Ka}, every element of 
$P(\Lambda)$ is of the form $\{y+k\delta \mid y \in \max P(\Lambda), k \in \mathbb Z_{ \geq 0}\}$.
Let $W$ denote the Weyl group
\[
W=T \rtimes  \mathring{W}
\] 
expressed as a semidirect product of 
a normal abelian subgroup $T$ by the finite Weyl group given by crossing out the first row and column of the Cartan matrix.  The elements of $T$ are transformations of the form

\[
t_\alpha(\zeta)=\zeta+r\alpha -((\zeta|\alpha)+\frac{1}{2}(\alpha|\alpha)r)\delta,
\]
\noindent where for type A, the weights $\a$ are taken from the $\mathbb Z$-lattice generated by $\a_1, \dots ,\a_{\ell}$, omitting $\a_0$.

\begin{defn} For any weight $\lambda$ in the set of weights  $P(\Lambda)$ for a dominant integral weight $\Lambda$, we let $\hub(\lambda)=(\theta_0,\dots, \theta_{e-1})$ be the hub of $\lambda$, where 
	\[
	\theta_i=\la \lambda,\alpha_i^\vee \ra
	\]
	\noindent The hub is the projection of the weight of $\lambda$ onto the subspace of the weight space generated by the fundamental weights \cite{Fa}.
\end{defn} 
\noindent We recall that for type $A$,  with null root $\delta=\alpha_0+\alpha_1+\dots+\alpha_{e-1}$, the symmetric product is 
\[
(\lambda|\alpha_i)=\langle \lambda,\alpha_i^\vee \rangle
\] 

By the ground-breaking work of Chuang and Rouquier \cite{CR}, the highest weight module $V(\Lambda)$ can be categorified.  The weight spaces lift to categories of representations of  blocks of cyclotomic Hecke algebras, the basis vectors lift to simple modules,  the Chevalley generators $e_i,f_i$ lift to restriction and induction functors $E_i,F_i$, and the simple reflections in the Weyl group lift to derived equivalences, which in a few important cases are actually Morita equivalences.
It follows from the work of Scopes \cite{Sc} and generalizations by Chuang and Rouquier \cite{CR}, that if we have weights at two ends of a string in a sense we will shortly make explicit, then acting on the blocks of the cyclotomic Hecke algebra  by the Weyl group will produce a Morita equivalence.

\section{EXTERNAL VERTICES OF THE REDUCED CRYSTAL}

The Kashiwara crystal, $B(\Lambda)$ \cite{K1} \cite{K2}, is a basis for $V(\Lambda)$ with some additional properties of which the only one of importance at the moment is the existence of operators $\tilde e_i$ and $\tilde f_i$ between basis elements. The set $P(\Lambda)$ can be taken as the set of vertices of a graph $\hat P(\Lambda)$ which we will call the reduced crystal, as in \cite{AS} or  \cite{BFS}.  Two vertices will be connected by an edge of residue $i$ if there are two basis elements with those weights connected in the Kashiwara crystal by $\tilde e_i$ or $\tilde f_i$. A finite set of vertices connected by edges of residue $i$ will be called an $i$-string.

\begin{defn}We say that an element $b$  of $B(\Lambda)$ is $i$-\textit{external} if for any  $i$ for which  $\theta_i > 0$, then 
  $\tilde  e_i(b)=0 $ and if   $\theta_i < 0$, then  $\tilde  f_i(b)=0 $ 
\end{defn}

\begin{lem}\label{defect}For a dominant integral weight 
	$\Lambda=a_0\Lambda_0+a_1\Lambda_1+\dots+a_\ell \Lambda_{\ell}$
	in a  crystal  $B(\Lambda)$, if $\eta=\Lambda-\alpha$ for $\alpha \in Q_+$ has a non-negative $i$ component $w=\theta_i$ in its hub, then
	\begin{enumerate}
		\item   The defect of $\lambda=\eta-k\alpha_i$ for $ 0 \leq k \leq w$ is
		$	\defect(\eta)+k(w-k)$.
		\item  The absolute values of the differences between the defects along the $i$-string from $\lambda$ to $\lambda'$ for $w$ odd are 
		\[
		w-1,w-3,\dots,4,2,0,2,4,\dots, w-3,w-1
		\]
		and for $w$ even are 
		\[
		w-1,w-3,\dots,3,1,1,3,\dots, w-3,w-1
		\]	
	\end{enumerate}	
\end{lem}

\begin{proof}
	\begin{enumerate}
		\item 	We simply compute the defect explicitly for the case of $\alpha_i$, using $(\alpha_i|\alpha_i)=2$, 
		\begin{align*}
			\defect(\lambda)&=(\Lambda|\alpha+k\alpha_i)
		-\frac{1}{2}(\alpha+k\alpha_i|\alpha+k\alpha_i)\\
		&=(\Lambda|\alpha)
		-\frac{1}{2}(\alpha|\alpha)+(\Lambda|k\alpha_i)
		-(\alpha|k\alpha_i)
		-\frac{1}{2}(k\alpha_i|k\alpha_i)\\
		&=\defect(\eta)+k((\Lambda|\alpha_i)-(\alpha|\alpha_i))-k^2\\
		&=\defect(\eta)+k((\Lambda-\alpha)|\alpha_i)-k^2\\
		&=\defect(\eta)+k(\la \Lambda-\alpha),\alpha_i^\vee\ra)-k^2\\
		&=\defect(\eta)+k(w-k)	
		\end{align*}
		\item For $0 \leq k \leq w-1$, we have the difference
		\begin{align*}
		|(k+1)(w-k-1)-k(w-k)|&=|kw-k^2-k+w-k-1-kw+k^2|\\
		&=|(w-1)-2k|
		\end{align*} 	
		
	\end{enumerate}
\end{proof}
In a reduced crystal for $\Lambda$, the $i$-string from $\Lambda$ has length $a_i$ . The weights $\Lambda-k\alpha_i$ necessarily lie in $\max(\Lambda)$ since  
$\Lambda-k\alpha_i+\delta$ cannot lie in $P(\Lambda)$ because the coefficient of $\alpha_j$  is negative for every $j \neq i$ and thus $k \alpha_i-\delta \notin Q_+$.
\begin{corl}\label{defect}
	In a  crystal with $\Lambda=a_0\Lambda_0 +\dots+ a_\ell\Lambda_\ell$, the defect of $\lambda=\Lambda-k\alpha_i$ for $ 0 \leq k \leq a_i$ is $k(a_i-k)$.
\end{corl}
\begin{proof}
	$\defect(\Lambda)=0$ and $\hub(\Lambda)=[a_0,\dots,a_\ell]$, so $w=a_i$.
\end{proof}

\begin{prop}\label{bound} Let $\Lambda$ be a dominant integral weight.  For any positive integer $d$,  a weight of defect $d$  is $i$-external if  $d \leq |\theta_i|$.
\end{prop}

\begin{proof}
 A weight of defect $d$ can only be an internal vertex of an $i$-string if there exists a neighboring vertex $\lambda'$ of lower defect $d'$.  Let $w$ be the absolute value of the  $i$-component of the hub of $\lambda'$. The absolute value of the $i$-component of the hub decreases when the defect increases,  and then $\lambda$ will have $i$-component of absolute value $w-2$, since weights increase or decrease by $2$ along an $i$-string. By Lemma 3.1.2, we have  $d-d'=w-1$. Since $d'\geq 0$ we have $d \geq w-1$. Since the $i$-hub $\theta_i$ of $\lambda$ is of absolute value $w-2$ , we have $d \geq |\theta_i|+1$. Thus, whenever $d \leq |\theta_i|$, the vertex must be $i$-external. 
\end{proof}

In order to show that for any $d$ there is a bound $N(d)$ on the degree $n$ such that every Morita equivalence class of blocks of $H^\Lambda_n$ appears in some degree $n < N$, we need to use the methods of \cite{BFS}.  In order to make our treatment self-contained, we review the necessary results.  The set $\max(\Lambda)$ is parametrized by an integral lattice $M$, which for type A corresponds to $\mathbb Z^\ell$.  
We let $m=(m_1,m_2,\dots,m_\ell)$ be an element of $M$ given in this parametrization.  We then form an element of the integral weight space
$P$ given by $\Lambda-(m_1\alpha_1+\dots+m_\ell \alpha_\ell)$.  By the main theorem of \cite {BFS}, there is an integer $s(m)$ such that 
\[
\eta_m=\Lambda-(m_1\alpha_1+\dots+m_\ell \alpha_\ell)-s(m)\delta \in \max(\Lambda)
\]
We choose our simple roots such that  $\alpha_0 = -\alpha_1-\alpha_2-\dots-\alpha_\ell+\delta$. Solving this for $\delta$ and substituting, we get
\[
\eta_m=\Lambda-(s(m)\alpha_0+(s(m)+m_1)\alpha_1+\dots+(s(m)+m_\ell) \alpha_\ell)
\] 
From this we obtain the content
\[
\cont(\eta_m)=(s(m),s(m)+m_1,\dots,s(m)+m_\ell)
\] 
and, setting $m_0=0$, with indices taken modulo $e=\ell+1$,  the hub is
\[
\hub(\eta_m)=(a_i-(2m_i-m_{i-1}-m_{i+1}))_{i=0}^\ell
\] 
Thus the components of the hubs are  linear transformations of the components of the content.

\begin{prop}
For any defect $d$, there is a minimal degree $N(d)$ such that every occurence of that defect in degree $n \geq N(d)$ is at the end of a string  to a vertex of lower degree.
\end{prop}
\begin{proof} By Prop. 3.1, it suffices to show that we can find some degree such that for every hub in that and all greater degrees, there is a component $\theta_i$ of the hub satisfying $\theta_i  \leq -d$. Then the vertex is an external vertex because $d \leq -\theta_i =|\theta_i|$ and lies at the high-degree end of the  $i$-string because $\theta_i$ is negative.

We first show that for any $d$, and any $i$, we can find an weight $\nu'_i$ in $\max(\Lambda)$ such that $\theta_j \leq -d$ for all $j \neq i$.  We first define $\nu_i$ to be $-(d+1)$ for $j \neq i$, and $r+\ell  (d+1)$ for component $i$, which is a weight of level $r$.  Let $\psi$ be the hub of $\Lambda-\nu_i$, which is of level $0$ since the level is additive.  By Proposition 3.6 of \cite{BFS}, in order for $\psi$ to lie in $Q$, we need
\[
\psi_1+2\psi_2+\dots+\ell \psi_\ell \equiv 0 \bmod e
\]
If it is congruent to $j$ with $j \neq 0$, then we can subtract $1$ from $\psi_j$ and add $1$ to $\psi_0$, getting $\psi'$ which is still of level $0$ but satisfies the condition.  We then define

\[
\nu'=\Lambda-\psi'.
\]
With the appropriate $\delta$-shift, this is the weight of some element of $\max(\Lambda)$ by Theorem 2.7 of \cite{BFS}.  

For any $i$, it may happen that $\nu_i'=\nu_i$.  If not, they differ in exactly two values of $\theta_t, 0 \leq t \leq \ell$.  The value of $\theta_0$ will be one less, and the value of one of the other $\theta_t$ for
$0 < t  \leq \ell$ will have $1$ added to it, since we subtracted $1$ from $\psi$ at that coordinate.  This means that if $\theta_0$ is negative, it is $-(d+1)$ or $-(d+2)$ and for any other $t>0$, if it is negative, it is $-(d+1)$ or $-d$.  In all cases, the negative values are less than or equal to $-d$.

The matrix $M$ with rows given by the entries in the $\nu_i$ has all rows summing to $r$. Then the matrix $M'$ with rows $\nu_i'$ also has rows summing to $r$, since we added a one and subtracted a one from two of the entries (possibly both the same entry). At this point we apply a theorem originally published by L\'evy in 1881 and rediscovered many times since, that a strictly diagonally dominant matrix is non-singular \cite{T}, so the rows of $M'$ determine a simplex. In the references in \cite{T}, Taussky lists 25 different papers in which proofs of this theorem have been given.

 If $i=0$, then  $\theta_0$ is determined by $-m_1-m_\ell$. By definition this is positive on $\nu'_0$ and negative at every corner of the $0$-face of the simplex, so that all the points on the opposite side of the set with $\theta_0 = -d$ have $0$-component $\theta_0 \leq -d$.  We conclude that the finite set of points for which none of the hub components is less than or equal to $-d$ must lie inside the simplex, which contains a finite number of points, with a finite maximum degree which we denote by $N(d)$. 

Any vertex $b$ of defect $d$ outside the simplex, since it has some $i$ with $\theta_i \leq -d$,  must be $i$-external, so after reflection in the $i$-string we get a weight $b'$ of lower degree.   If $b'$ is not yet in the simplex, there will be a $j$ with $\theta_j \leq -d$.  We continue in this fashion until we get to a degree lower than $N(d)$.

\end{proof}

\begin{corl*} There are only a finite number of Morita equivalence classes of blocks of cyclotomic Hecke algebras $H^\Lambda_n$, $n \in \mathbb Z_+$.
	\end{corl*}

\begin{proof}
In \cite{CR}, not only did Chuang and Rouquier prove the categorification theorem for the cyclotomic Hecke algebras $H^\Lambda_n$, they also prove an adjunction theorem.  If $\mu$ is a basis element of the crystal $B(\Lambda)$ with weight $\lambda$  which lies at the end of an $i$-string in the reduced crystal, and $S$ is the corresponding simple module in the cyclotomic Hecke algebra, then there is an idempotent $e$ in the blocks (of the same defect) corresponding to the image under the action of the Weyl group element $s_i$ which cuts out a block corresponding to $\mu$ under the derived equivalence.  The induced module that one gets by going down the string is in fact $w!$ copies of this block, so we get $w!$ copies of each simple.  A derived equivalence which gives a one-to-one correspondence between simples is, by a result of Linckelmann \cite{Lin}, a Morita equivalence.   To apply the Linckelmann result, we must know that the cyclotomic Hecke algebras of type A are symmetric.  This is apparently well-known.  Malle and Mathas give a sketch of the proof in their introduction \cite{MM}.	
\end{proof}

\begin{figure}[ht]
\centering
\includegraphics[scale=0.5]{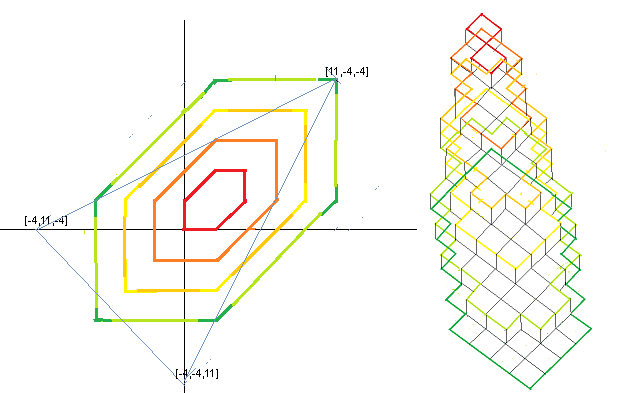}
\captionsetup{labelformat=empty}
\caption{Figure 1.  The reduced crystal for $ e=3, \Lambda=\Lambda_0+\Lambda_1+\Lambda_2$}
\end{figure}
In Figure 1, on the left is the two dimensional lattice $M = \{m=(m_1,m_2)|m_i \in \mathbb Z\}$ corresponding to elements of $\max(\Lambda)$.  The latice is  embedded in the real plane, with the $m_1$-axis horizontal and the $m_2$-axis verical. The hexagonal contours from inside outward contain points with  $\delta$-shifts $s(m)$ of $0$, $1$, $2-3$, and $4-5$. Each point is also associated to a hub by the formula $\hub(m)=(1+m_1+m_2,1- 2m_1+m_2,1 +m_1-2m_2)$, which sums to $3$, and we have labelled the points with the hubs.   Outside of the superimposed triangle, every hub has a component $\leq -4$.

In Figure 1, on the right,  we give a truncation of  the reduced, three-dimensional crystal, in which edges bounding a region with fixed $\delta$-shift are emphasized. Horizontal lines in the reduced crystal  correspond to fixed degree.

\section{EXTERNAL VERTICES OF THE REDUCED CRYSTAL FOR $e=2$}

  In the case of $e=2$,  the lattice is in one-to-one correspondence with  the integers and we can give an explicit bound  $N'(d)$ on the degree such that all Morita equivalence classes occur for degrees less than $N'(d)$. Since $e=2$, then $\ell=e-1=1$. For the simple roots, we take
	\[
	\alpha_0=2\Lambda_0-2\Lambda_1+\delta
	\]
	\[
	\alpha_1=-2\Lambda_0+2\Lambda_1
	\]
	
We begin by giving a description of the set $\max(\Lambda)$, with weights, hubs, defect and contents. 

We will make extensive use of \cite{BFS}, in which every weight in $\max(\Lambda)$ is described as a transformation by an element of the abelian subgroup $T$ of the Weyl group $W$ of one of the elements of a set $N$ of representatives of $T$-orbits.  In order to give a very explicit description of the set $N$ in the case $e=2, M \cong \mathbb Z$, we will use an different than normal representation of the integers in the form $m=qr+u$, chosen so that all the hubs with all components non-negative will correspond to integers $m$ with $q=0$.  In order to acheive this, we will have to take the unusual step of allowing negative remainders $u$ in certain cases, as described in the following lemma.

\begin{lem}
Let $e=2$ and $\Lambda=a_0\Lambda_0+a_1\Lambda_1$, with $r=a_0+a_1$.  Set $b_0=\lfloor \frac{a_0}{2} \rfloor$.
Write the integers in the form $m=qr+u$, where $-b_0 \leq u \leq r-b_0-1$.
Set 
\[
\eta_m=(a_0+2m)\Lambda_0+(a_1-2m)\Lambda_1-t(m)\delta,
\]
where 
\[
t(m)=\max(-u,0,u-a_1)-q(a_1-2u)+q^2r.
\]
Then
\begin{enumerate}
	\item $\hub(\eta_m)=[a_0+2m,a_1-2m]$,
	\item $\max(\Lambda)=\{\eta_m|m \in \mathbb Z \}$
	\item $\defect(\eta_m)=max(-u(a_0+u),u(a_1-u),(u-a_1)(r-u))$,
	\item $\cont (\eta_m)=(s(m),s(m)+m )$
	\item $\deg(\eta_m)=2s(m)+m$.
\end{enumerate}
\end{lem}

\begin{proof} The weights $\eta_m$ are given numerically, and the main point of the proof is to show that the weights given by these formulae actually lie in 
$P(\Lambda)$ and, in fact, give all the elements of $\max(\Lambda)$. That done, we also give the values of the various other invariants, where the calculation of the content shows that they coincide with the $\eta_m$ defined in \S3.  

To provide some orientation before starting the proof, note that when $m=0$, we have $q=0,u=0,$ so $\eta_0=a_0\Lambda_0+a_1\Lambda_1+0 \cdot \delta=\Lambda$, which is indeed an element of $\max(\Lambda)$. Then $\eta_{-1}=\Lambda-\alpha_0$ and $\eta_1=\Lambda-\alpha_1$ also lie in $\max(\Lambda)$.  
	
	\begin{enumerate}
    	\item In the formula for $\eta_m$, the coefficient of $\Lambda_0$ is $a_0+2m$ and the coefficient of $\Lambda_1$ is $a_1-2m$. The hub is defined to be this  pair of coefficients of the $\Lambda_i$ in $\eta_m$.
    	\item  
  Let $ b_1=\lfloor \frac{a_1}{2} \rfloor$. As shown earlier, all the vertices on the two strings going out from the highest weight  $\Lambda$ must lie in $\max(\Lambda)$.
  The set of dominant integral weights $N_0=P_+\cap  \max(\Lambda)$ form a subset of  these two strings. The $0$-string has length $a_0$ and the $1$-string has length $a_1$. The weights in $N_0$  have hubs $[a_0-2b_0, a_1+2b_0],\dots,[a_0,a_1],\dots, [a_0+2b_1,a_1-2b_1]$, and are obtained from $\Lambda$ by adding  $u\alpha_i$ out to half of each $i$-string.
  
  The finite Weyl group $\mathring{W}$ has two elements, the identity and $s_1$. As in \cite{BFS},  let $N$ be the orbit of $N_0$ under  $\mathring{W}$. 
  We wish first to show that the elements of $N$ all have the form $\eta_m$ for some $m \in M$.
  \begin{itemize}
  	\item $-b_0 \leq m \leq -1$. Here $q=0, m=u<0$ and $t(m)=-u=|u|$. The weights of the form $\Lambda -k\alpha_0$ have the hub $[a_0-2k,a_1+k]$ and we have subtracted $k$ copies of $\delta$, one for each copy of $\alpha_0$. since $k=|u|$, this gives just the correct value of $t(m)$, so $\eta_u=\Lambda+u\alpha_0$.
  	
  	\item $0 \leq m \leq b_1$.  We still have $q=0$ and $m=u \geq 0$. The weights $\Lambda-u\alpha_1$ have the same hub as $\eta_u$ and $\delta$-shift $0$.  When we calculate $t(m)$ using $q=0$, we get $t(m)=\max(-u,0,u-a_1)=0$.
  	
  	\item $b_1 < m < a_1$.    When we act on the dominant integral weights from the $1$-string by $s_i$, we get the remaining elements of the $1$-string.   Each of these weight is the result of a reflection $s_1$ acting on one of the previous $m$. The weight
  	$\Lambda-m\alpha_1$ had hub $[a_0+2m, a_1-2m]$, which coincides with the hub of $\eta_m$ given in (1), and there are no copies of $\delta$ because here are no copies of $\alpha_0$. It remains to show that $t(m)=0$.  We still have $q=0$ since $m \leq a_1 \leq r-b_0-1$, so $m=u$ and $\max(-u,0,u-a_1)=0$ so we still have $t(m)=0$, as needed.
  	
  	\item $a_1 \leq m \leq a_1+b_0$. 
  	To fill out $N$, we act on the dominant integral weights in the half $0$-string by $s_1$. This give another $b_0+1$ elements of $N$ which lie at the end of $1$-strings.
  	For $\eta_{u}$ with  $-b_0 \leq u \leq 0$ we have $s(u)=t(u)=|u|$.  The hub of $s_1\eta_{u}$ is
  	\[
  	[a_0+2u,a_1-2u]-(a_1-2u)[-2,2]=[a_0+2a_1-2u,-a_1+2u]
  	\]
  	  \noindent which is the hub of $\eta_m$ for $m=a_1-u=a_1+|u|$. The $\delta$-shift of $s_1\eta_{u}$ is also $|u|$ because the action of the Weyl group element $s_1$ does not change $\delta$. To show that $\eta_m=s_1\eta_{u}$, it remains to show that $t(m)=|u|$.
  	  If $-b_0<u\leq 0$ or $u=-b_0$ but $a_0$ is odd, then $q=0$, so 
  	  $t(m)=\max(-m,0,m-a_1)=m-a_1=|u|.$ If $u=-b_0$ but $a_0$ is even, then $m=a_1+a_0-b_0$
  	  =$r-b_0$, so
  	  \begin{align*}
  	  t(m)=&\max(b_0,0,-b_0-a_1)+(a_1-2(-b_0))+r\\
  	  =&b_0-r+r\\
  	  =&b_0=|u|\\
  	  \end{align*}

    \end{itemize}

\noindent Claim:  $\eta_{qr+u}=t_{-q\alpha_1}(\eta_u), -b_0 \leq u \leq r-b_0-1$.

To demonstrate the claim, we simply calculate the action of $t_{-q\alpha_1}$:
\begin{align*}
t_{-q\alpha_1}(\eta_u)&=\eta_u+r(-q\alpha_1)-((\eta_u\mid -q\alpha_1)+
\frac{1}{2}q^2(\alpha_1|\alpha_1)r)\delta\\
&=(a_0+2u)\Lambda_0+(a_1-2u)\Lambda_1-max(-u,0,a_1-u)\delta\\
&-qr\alpha_1-	((a_1-2u)(-q)+q^krk,2r)\delta\\
&=(a_0+2(qr+u))\Lambda_0+(a_1-2(qr+u))\Lambda_1\\
&-(\max(-u,0,a_1-u)-q(a_1-2u)+q^2r)\delta\\
&=\eta_{qr+u}		
\end{align*}

This shows that all the $\eta_m$ lie in $\max(\Lambda)$.  However, since $q$ can take any integer value, it also shows that every element of $\max(\Lambda)$ is of the form $\eta(m)$ for some $m$. 

Since defect is preserved by the action of $T$, this also shows that the defect is independent of $q$.

Following 
\cite{BFS}, we want to shrink $N$ to a fundamental region $\bar N$ by taking a single representative of each $T$ orbit. If $a_0$ is odd, then $N=\bar N$ because, as mentioned above, $r-b_0-1=a_1+b_0$ so no element of $N$ can be written with nonzero $q$.  If, however, $a_0$ is even, so that $2b_0=a_0$, then $a_1+b_0=r-b_0$, so 
\[
\eta_{a_1+b_0}=t_{-\alpha_1}(\eta_{-b_0})
\]
and we will not include $\eta_{a_1+b_0}$ in $\bar N$.

\item  The defects are given by  Corollary \ref{defect} as $-u(a_0+u)$ when $u$ in non-positive and by $u(a_1-u)$ when $u$ is non-negative, so we have defects $\eta_u$ for $-b_0 \leq u \leq b_1$ given by $\max(-u(a_0+u), u(a_1-u))$, since each of these parabolas is non-positive where the other is non-negative. Finally, for both $u$ negative and $u$ positive, we have $(u-a_1)$ negative and $r-u$ positive. Thus in the formula for the defect in the lemma, in these two cases we choose one of the first two alternatives.  

  Since action by the Weyl group preserves defect, we obtain the defects for the second half of the $1$-string, $b_1 <m \leq a_1$, where as reflections of $\eta_u$ for $0 \leq u \leq b_1$.
\[
\defect(\eta_{a_1-u})= \defect(\eta_{u})=\max(-u(a_0-u),u(a_1+u))=u(a_1+u)
\]
In this third case, the other two possibilities in the formula are both negative. We have $-u(a_0+u)$ is negative because $u$ is positive, and $(u-a_1)(r-u)$ is negative because 
$u<a_1\leq r$. 

Finally, we have the fourth case, of the $a_1+1 \leq m \leq a_1 + b_0$, so that if we define $u=a_1-m$, then  $-b_0 \leq u \leq -1$. For $u$ in this range, $\defect(\eta_u)=-u(a_0+u)$. Substituting, we get a defect equal
to $(m-a_1)(a_0+a_1-m)$, which is non-negative.  In the given range for $m$, the first two possibilities in the formula for the defect in the lemma, $-m(a_0+m)$ and $m(a_1-m)$, are negative, so if we substitute $r=a_0+a_1$ we get the desired formula in this fourth case. 

These four cases cover all possibilities for $q=0$, and we showed in (2) above that the defect is independent of $q$, so this formula gives the defect for all cases.
\[
\defect(\eta_{u})=\max(-u(a_0+u),u(a_1-u),(u-a_1)(r-u)), -b_0 \leq u \leq r-b_0-1
\] 
since $r-b_0-1 = a_1+b_0$ if $a_0$ is odd and $r-b_0-1<a_1+b_0$ if $a_0$ is even.

 \item Letting $\cont(\eta_m)=(c_0,c_1)$, then by the definition of $\alpha_0$, we get $c_0=s(m)$. Substituting into the formula
 \[
 \eta_m=\Lambda-c_0\a_0-c_1\a_1
 \] 
 and projecting only the first component of the hub gives
 \[
 a_0+2m=a_0-2s(m)+2c_1
 \]
 Eliminating $a_0$ and dividing by $2$, we conclude that $c_1=s(m)+m$, so
 \[
 \cont(\eta_m)=(s(m),s(m)+m)
 \]
\item The degree is the sum of the components of the content.

\end{enumerate}
 \end{proof}

\begin{examp}
	As an example, consider the case of $a_0=2$ and $a_1=1$.  The reduced crystal $\hat P(\Lambda)$ is given in Figure 2. To illustrate the lemma above, we give the results in tabular form.
	
	\begin{center}
		\begin{tabular}{ | l || c | c | c | c | c | c | c | }
			\hline
			m &-3 & -2 & -1 & 0 & 1 & 2 & 3 \\ \hline \hline
			Hub& [-4,7]&[-2,5]&[0,3]& [2,1]& [4,-1] & [6,-3] & [8,-5] \\ \hline
			Defect &0&0&1& 0& 0 & 1 & 0 \\ \hline
			Content &(4,1)&(2,0)&(1,0)&(0,0)&(0,1)&(1,3)&(2,5) \\ \hline
			Degree &5&2&1& 0& 1 & 4 & 7 \\ \hline 
			\hline
		\end{tabular}
	\end{center}
\end{examp}

\noindent \begin{figure*}[h]\label{reduced}
	\centering
	\includegraphics[scale=0.7]{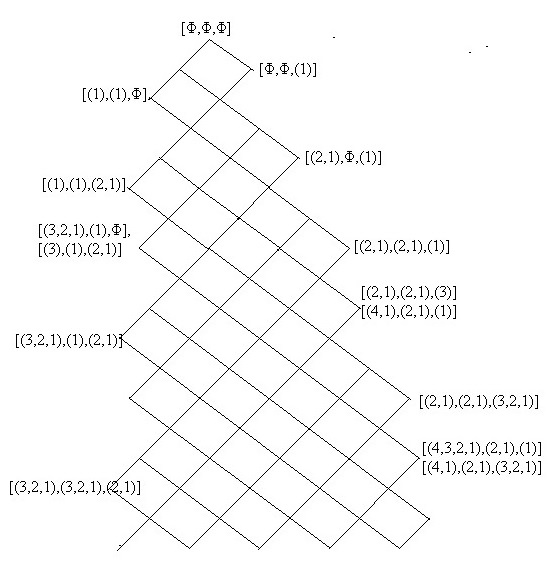}
	\captionsetup{labelformat=empty}
	\caption{Figure 2: $e=2,\Lambda=2\Lambda_0+\Lambda_1$}
\end{figure*}

We now use the information we have accumulated in the case of $e=2$ to give an explicit upper estimate $N'(d)$ for $N(d)$, the bound after which the vertices are all external.

\begin{prop}\label{bound}  Let
	 $\Lambda$ be a dominant integral weight. Assume $r>1$.
	  For any positive integral $d$, set 
	\[
	m_0=-\left \lceil \frac{d+a_0}{2} \right \rceil 
	\]
	\[
    m_1=\left \lceil \frac{d+a_1}{2} \right \rceil 
	\]
	
Then the weights $\eta$ of defect $d$ are all external whenever $\deg(\eta) \geq N'(d)$ for 
\[
N'(d)= \max(2s(m_0)+m_0,2s(m_1)+m_1)
\]
\end{prop}

\begin{proof} 
 By Prop 3.1, a weight $\eta$  of defect $d$ is an $i$-external  vertex of an $i$-string if  $d \leq |\theta_i|$. We want this to hold for both $i$. 
 
 All vertices of defect $0$ are external by Prop. 3.1 since the sum of the components of the hub is the positive integer $r$.  If the weight has only non-negative components and is not of defect zero, then it corresponds under the action of the Weyl group to a weight of higher degree with the same defect and a negative value of one of the $\theta_i$, so we assume henceforward that one of the $\theta_i$ is negative.  The possible hubs have the form $[a_0+2m,a_1-2m]$. If $\theta_0$ is negative, then $a_0 <-2m$, while if $\theta_1$ is negative, then $a_1 < 2m$
  The negative $\theta_i$ always has the smaller absolute value since the $\theta_i$ sum to a positive integer $r$. 
  
  For each side of the reduced crystal, we want to find the last weight $\eta$ of defect $d$ which is still internal, which by Prop. 3.1 means that $d > |\theta_i|$. Then we will add one to the degree.
  If we are at the high degree end of a $0$-string, which is on the left in the sample reduced crystal in Figure 2, we have chosen the non-positive $m_0$ so that $0 \leq -m_0-1 <  \frac{d+a_0}{2} \leq -m_0$ so  we get
 \[
 |a_0+2m_0+2| < d \leq |a_0+2m_0| =2|m_0|-a_0
 \]
 Similarly, from the definition of $m_1$, we have $0 \leq m_1-1 <  \frac{d+a_1}{2} \leq m_1$, so 
  \[
|a_1-2m_1+2| < d \leq |a_1-2m_1|=2m_1-a_1
\]
The integers $m_0,m_1$ are the  integers of smallest absolute value for which the equation holds.

We now convert this into an inequality for the degree.  We substitute for $t(m)$ to get
\[
deg(\eta_{m_0})= 2t(m_0)+m_0  	
\] 
\[
deg(\eta_{m_1})= 2t(m_1)+m_1  	
\]
All weights with defect $d$ are obtained by Weyl group reflections from $\delta$-shifts of weights in $N$. The absolute value of the
negative hub at $m_0+1$ or $m_1-1$ is no more than $1$ less than $d$, so since we have assumed $r>1$, the positive value must be greater than or equal to $d$.  This means that after action by the Weyl group reflection which converts the positive to negative, the absolute values of both of the $\theta_i$ will be greater than or equal to $d$, and the weight will be external by Corollary \ref{defect}. Since we have found the last possible internal weight of defect $d$ on each side, the maximum of the next largest weight must be the first degree after which all weights of defect $d$ are external, giving 

\[
N'(d)= \max(2t(m_0)+m_0,2t(m_1)+m_1)
\]
as desired
\end{proof}

Let us check the sharpness of the bound  $N'(d)$ given in Proposition \ref{bound}. Let $N(d)$ be the sharp bound, the degree from which the vertices of defect $d$ are $i$-external for both $i=0$ and $i=1$ and we compare the two bounds in tabular form for
Example 1.  

	\begin{center}
	\begin{tabular}{ | l || c | c | c | c | c | c | c | c| c | c | }
		\hline
		d &0 & 1 & 3 & 4 & 6 & 7 & 9& 10 &12&13  \\ \hline \hline
		$m_0$ &0&-2&-3&-4&-5&-5&-6&-7&-8&-8 \\ \hline
		$m_1$ &0&2&3&3&4&5&6&6&7&8 \\ \hline  
        $\deg(\eta_{m_0})$&0&2&5&10&15&15&22&31&40&40 \\ \hline
        $\deg(\eta_{m_1})$&0&4&7&7&12&19&26&26&35&46  \\ \hline
        $N(d)$ &0&2&5& 7& 12 & 15 & 22&26 &35&40 \\ \hline
		$N'(d)$ &0& 4&7&10&15&19&26&31&40&46  \\ \hline \hline		
	\end{tabular}
\end{center}
\noindent The actual bound $N(d)$ turns out, in this example, to be the smaller of the two degrees of which we, for safety's sake, have taken the maximum to be our estimate $N'(d)$. Since the difference seems to be polynomial in $d$, this should not make much difference computationally.

\end{document}